\documentclass[11pt,reqno]{amsart}
\usepackage{eurosym}
\usepackage{amsmath}
\usepackage{amsfonts}
\usepackage{amssymb}
\usepackage{amsthm}
\usepackage{amscd}
\usepackage{amsgen}
\usepackage{latexsym}
\usepackage{url}
\usepackage{mathtools}
\usepackage{needspace}
\usepackage[breaklinks=true,colorlinks=true,citecolor=blue,linkcolor=blue,urlcolor=red]{hyperref}
\everymath{\displaystyle}
\providecommand{\U}[1]{\protect\rule{.1in}{.1in}}
\newtheorem{theorem}{Theorem}[section]

\newtheorem{lemma}{Lemma}[section]

\newtheorem{remark}{Remark}[section]
\newtheorem{proposition}{Proposition}[section]
\numberwithin{equation}{section}
\allowdisplaybreaks
\catcode`\@=11
\@namedef{subjclassname@2010}{  \textup{2010} Mathematics Subject Classification}
\catcode`\@=12

\begin{document}
\title[Sign-changing solutions...]{\textbf{Sign-changing solutions for a \textit{Yamabe} type problem. }}
\author{Mohamed Bekiri }
\address{Mohamed Bekiri, Laboratory of Geomatics, Ecology and Environment (LGEO2E), Department of Biology, Faculty of Natural and Life Sciences, University Mustapha
Stambouli of Mascara, Algeria.}
\email{mohamed.bekiri@univ-mascara.dz}
\author{Mohammed Elamine Sebih }
\address{Mohammed Elamine Sebih, Laboratory of Geomatics, Ecology and Environment (LGEO2E), Department of Mathematics, Faculty of Exact Sciences, University Mustapha
Stambouli of Mascara, Algeria.}
\email{ma.sebih@univ-mascara.dz}

\thanks{}
\date{\today }
\subjclass[2020]{Primary 53A99, 58J05, 53C21}
\keywords{Sign-changing solutions, Yamabe type operator, Compact
manifold with boundary, variational method }
\dedicatory{ }

\begin{abstract}
In this paper, we investigate the
existence of sign-changing solutions to a critical elliptic equation involving a Yamabe type operator on a compact manifold
with boundary. The existence result is assured under some geometric
conditions.
\end{abstract}
\maketitle
%\tableofcontents

\section{Introduction and statement of results}

Given a compact Riemannian manifold $(M,\,g)$ with smooth boundary of dimension $n\geq 3$, $R_{g}$ its scalar curvature. In 1960, Yamabe~\cite{Yamabe1960} announced that for all compact Riemannian $n$-manifold
$(M,\, g)$ there exists a conformal metric $\tilde{g}$ to $g$ such that the scalar curvature $R_{\tilde{g}}$ is constant. In
other words, if we consider the conformal deformation in the form $\tilde{g}=u^{\frac{4}{n-2}}g$ (with $u\in C^{\infty}(M),\: u>0$), the scalar curvature $R_{\tilde{g}}$ satisfies the equation 
\begin{equation}
    \Delta _{g}u+\frac{n-2}{4\left( n-1\right) }R_{g}u=\frac{n-2}{4\left(
n-1\right) }R_{\tilde{g}}u^{2^{\star}-1}, \label{eq1-1}
\end{equation}
where $\Delta _{g}:=-\operatorname{div}_{g}(\nabla )$ is the Laplace-Beltrami operator, and $2^{\star}=\frac{2n}{n-2}$ is the critical Sobolev exponent. The Yamabe problem is reduced to solving the equation 
\begin{equation}
    \Delta _{g}u+\frac{n-2}{4\left( n-1\right) }R_{g}u=\lambda u^{2^{\star}-1}, \label{eq1-2}
\end{equation}
where $\lambda$ is a real parameter and the solution $u$ must be smooth and strictly positive.

This problem was completely solved by the combined efforts of Yamabe~\cite{Yamabe1960}, Trudinger~\cite{Trudinger1968}, Aubin \cite{Aubin1976} and Schoen~\cite{Schoen1984}. For more details, we refer the interested reader to the classical paper of Lee-
Parker~\cite{Lee-Parker1987}.

We note that, if $u$ is a sign-changing solution to problem \eqref{eq1-2}, then $\tilde{g}=|u|^{\frac{4}{n-2}}g$ is not a metric, as $\tilde{g}$ is not smooth and it vanishes on the set of zeros of $u$. 

A natural and interesting generalization of the equation \eqref{eq1-2} is to consider
\begin{equation}
    \Delta _{g}u+au=\lambda f |u|^{2^{\star}-2}u,\label{eq1-3}
\end{equation}
where $a$, $f$ are smooth functions, and $\lambda$ is a real parameter. The problem of finding sign-changing solutions of this equation has been studied by several authors. We cite for instance Atkinson-Brézis-Peletier~\cite{Atkinson-Brezis-Peleier1990}, Cao-Noussair~\cite{Cao-Noussair1995}, Djadli-Jourdain~\cite{Djadli-Jourdain2002}, Hebey-Vaugon~\cite{Hebey-Vaugon1994}, and Holcman~\cite{Holcman2001} to mention only few of many recent publications.

In this work, for a smooth compact Riemannian manifold $(M,g)$ of dimension $n>3$, with smooth boundary $\partial M\neq \emptyset$, we investigate the existence of sign-changing solutions for the following Yamabe-type problem 
\begin{equation}
\left\{ 
\begin{array}{lll}
-\operatorname{div}_{g}(a\nabla u)+bu=\lambda f|u|^{2^{\sharp }-2}u & \text{ in }&M, \\[0.5em] 
u=\phi & \text{ on }&\partial M,%
\end{array}%
\right.   \label{eq1-4}
\end{equation}%
where $a,\, b,\, f\in C^{\infty}(M)$, $a,\, f$ are positive functions on $M$ and $2^{\sharp }=\frac{2n}{n-2}$ is the critical Sobolev exponent and the boundary data $\phi\in C^{\infty }(\partial M)$ is a sign-changing function. More precisely, we look for
conditions on the functions $a,\, b$, and $f$ such that the problem \eqref{eq1-4} has a sign-changing solution.

Our results extend the work of Holcman~\cite{Holcman2001} where an elliptic problem involving the conformal Laplacian is considered. Our main result is stated in the following theorem:
\begin{theorem}
\label{Theorem1-1} Let $\left( M,g\right) $ be a compact Riemannian manifold
of dimension $n>3$ with smooth boundary $\partial M\neq\emptyset$. Assume that $a,\,b,\, f\in
C^{\infty}(M)$, with $a>0,\, f>0$ on $M$. Let $x_{0}\in \operatorname{Int}(M)$ be a point in the interior of $M$ such
that $f\left(  x_{0}\right)  =\max_{M}f$. Under the following assumptions

\begin{itemize}
\item The operator $-\operatorname{div}_{g}(a\nabla )+b$ is coercive,

\item
    at a point $x_{0}$, we have
\begin{equation}\label{eq1-5}
    (n-2)(n-4)\frac{\Delta_{g}f(x_{0})}{f(x_{0})}-2(n-2)R_{g}(x_{0})+8(n-1)\frac{b(x_{0})}{a(x_{0})}-(n^2-4)\frac{\Delta_{g}a(x_{0})}{a(x_{0})}<0,
\end{equation}
\end{itemize}
there exist a positive real number $\lambda$ and a nontrivial solution $%
u=w+h\in H_{0}^{1}\left( M\right) \cap C^{2,\alpha}\left( M\right) $ for
the equation \eqref{eq1-4}, where $w$ is the minimizer of the functional $I$
defined on $H_{0}^{1}\left( M\right) $ by 
\begin{equation*}
    I\left( w\right) =\int_{M}\left(a\vert
\nabla w\vert^{2}+bw^{2}\right)\,dv_{g},
\end{equation*}
under the constraint $\int_{M}f\left\vert w+h\right\vert
^{2^{\sharp}}\,dv_{g}=\gamma$, and $h$ is the unique solution of the linear problem 
\begin{equation*}
\left\lbrace 
\begin{array}{ll}
-\operatorname{div}_{g}(a\nabla h)+bh=0 & \text{ in } M, \\[0.5em] 
h=\phi  & \text{ on } \partial M.%
\end{array}
\right. 
\end{equation*}
Moreover, when the boundary data $\phi$ changes sign, then $u$ changes
sign.
\end{theorem}
The proof of Theorem \ref{Theorem1-1} is based on arguments introduced by Yamabe~\cite{Yamabe1960}. The strategy begins by constructing a minimizing sequence of solutions for the following family of sub-critical problems:
\begin{equation}
\left\{ 
\begin{array}{lll}\label{eq1-6}
-\operatorname{div}_{g}(a\nabla u)+bu=\lambda f|u|^{q-2}u & \text{ in }&M,\\[0.5em]
u=\phi _{1} & \text{ on }&\partial M,%
\end{array}%
\right.
\end{equation}
where $q\in \left( 2,2^{\sharp }\right) $ is the sub-critical exponent.%
\newline
In a second step, we prove that under suitable geometric conditions, this minimizing sequence converges to a nontrivial smooth solution of the critical problem \eqref{eq1-4}, as the sub-critical exponent $q$ tends to $2^{\sharp }$.

The paper is organized as follows. In section Sec.~\ref{sec:sec2}, we recall some definitions and preliminary results that will be used throughout the paper. Sections Sec.~\ref{sec:sec3} and Sec.~\ref{sec:sec4} are devoted to the study of 
the sub-critical and critical equations respectively. Finally, section Sec.~\ref{sec:sec5} is reserved to test functions.

%\section{Preliminaries}
\section{Terminology and general notations}
\label{sec:sec2} In this section, we recall some results and properties that will be needed in the sequel. Let $\left( M, g\right)$ be a smooth compact Riemannian manifold of
dimension $n \geq3$, with boundary. The Sobolev space $H_{0}^{1}\left(M\right) $ is defined as the completion of $C^{\infty}\left( M\right) $ (The
space of smooth functions) with respect to the following equivalent norm:
$
\Vert . \Vert_{H_{0}^{1}\left( M\right)}$ (see Aubin~\cite{Aubin1998})

\begin{equation}
\left\Vert u\right\Vert _{H_{0}^{1}\left( M\right) }^{2} =\Vert \nabla
u\Vert_{2}^{2}+\Vert u\Vert_{2}^{2}.  \label{eq2-1}
\end{equation}

\begin{remark}
The Sobolev embedding theorem asserts that the inclusion $H_{0}^{1}\left( M\right)\subset L^q \left( M\right)$ is
\begin{itemize}
    \item continuous for all $1<q\leq 2^{\sharp}=\frac{2n}{n-2}$,
    \item compact for all $1<q\neq 2^{\sharp}$.
\end{itemize}
\end{remark}
Now, we define $K_{0}>0$, to be the best constant in the Euclidean Sobolev inequality 
\begin{equation*}
\left\Vert u\right\Vert _{2^{\sharp}}^{2}\leq K_{0}\left\Vert \nabla
u\right\Vert _{2}^{2}.
\end{equation*}
The exact value of this constant was computed by Talenti~\cite{Talenti1976}. He showed that 
\begin{equation}
\frac{1}{K_{0}} = \frac{n(n - 2)}{4} \, \omega_n^{2/n},
\label{eq2-2}
\end{equation}
where $\omega_{n}$, denotes the volume of the unit $n$-sphere $(\mathbb{S}%
^{n},h)$, endowed with its standard metric.

We say that the operator $ P_g:=-\operatorname{div}_{g}(a\nabla )+b$ defined on $H_{0}^{1}\left( M\right)$, is coercive if there exists $\Lambda >0$, such that for all $u\in
H_{0}^{1}\left( M\right) $
\begin{equation*}
\int_{M} uP_{g}(u)\,dv_{g}\geq\Lambda\left\Vert u\right\Vert
_{H_{0}^{1}\left( M\right) }^{2}, 
\end{equation*}
where 
\begin{equation*}
    \int_{M} uP_{g}(u)\,dv_{g}=\int_{M}\left(a\vert
\nabla u\vert^{2}+bu^{2}\right)\,dv_{g}.
\end{equation*}
We also recall the following Sobolev inequality, that will play a key role in several estimates. This inequality was established by Hebey and Vaugon~\cite{Hebey-Vaugon1995}.
\begin{lemma}
\label{lemme2-1} Let $\left( M,g\right) $ be a smooth compact Riemannian
manifold of dimension $n>3$ with smooth boundary. Then for any $\varepsilon>0,$
there exists $B_{\varepsilon}\in\mathbb{R}$ such that for all $u\in
H_{0}^{1}\left( M\right) $ one has 
\begin{equation}
\left\Vert u\right\Vert _{2^{\sharp}}^{2}\leq\left( K_{0}
+\varepsilon\right) \left\Vert \nabla u\right\Vert _{2}
^{2}+B_{\varepsilon}\left\Vert u\right\Vert _{2}^{2}.
\label{eq2-3}
\end{equation}
\end{lemma}

\section{Construction of Subcritical Solutions}

\label{sec:sec3} %\section{Construction of a solution to sub-critical case}
In this section, we prove the existence of solutions to the subcritical problem \eqref{eq1-6}. Firstly, we state the following useful lemma, which extends the boundary data $ \phi$ to a unique smooth solution of the homogeneous problem associated to \eqref{eq1-4}, defined on the whole manifold $M$.
\begin{lemma}
Let $(M, g)$ be a smooth compact Riemannian manifold of dimension $n > 3$ with smooth boundary. Assume that the operator $-\operatorname{div}_g(a \nabla) + b$ is coercive. Then the following problem:
\begin{equation} 
\left\{
\begin{array}{lll}
-\operatorname{div}_g(a \nabla h) + b h = 0 & \text{in } &M, \\[0.5em]
h = \phi & \text{on }&\partial M,
\end{array}
\right.\label{eq3-1}
\end{equation}
has a unique solution $h \in C^{2,\alpha}(M)$ for some $\alpha \in (0,1)$.
\end{lemma}
\begin{proof}
The proof is classical. Existence and uniqueness follow from Lax-Milgram theorem, since the operator is coercive and the boundary condition is compatible. Furthermore, standard elliptic regularity theory implies that $h \in C^{2,\alpha}(M)$ for some $\alpha \in (0,1)$.
\end{proof}

We look for solutions to \eqref{eq1-4} in the form $u= w + h$, where $h$ is the solution to the linear problem \eqref{eq3-1} and $w$ solves
\begin{equation}
\left\{ 
\begin{array}{lll}
-\operatorname{div}_{g}(a\nabla w)+bw=\lambda f|w+h|^{2^{\sharp }-2}(w+h) & \text{ in }&M, \\[0.5em] 
w=0 & \text{ on }&\partial M.
\end{array}%
\right.   \label{eq3-2}
\end{equation}%
Given $q \in (2, 2^{\sharp})$, we consider the sub-critical problem associated with the critical problem~\eqref{eq3-2}:

\begin{equation}
\left\{ 
\begin{array}{lll}
-\operatorname{div}_{g}(a\nabla w)+bw=\lambda f|w+h|^{q-2}\left( w+h\right)  & \text{ in }&M, \\[0.5em] 
w=0
& \text{ on }&\partial M.%
\end{array}%
\right.   \label{eq3-3}
\end{equation}%
We define the energy functional $I$ on $H_{0}^{1}(M)$ by 
\begin{equation*}
    I(w)=\int_{M}\left(a\vert
\nabla w\vert^{2}+bw^{2}\right)\,dv_{g}.
\end{equation*}
We denote by $\mu_{\gamma, q}$ the minimum of the functional $I$ over the constraint set
\begin{equation*}
\mathcal{H}_{\gamma, q} = \left\{ w \in H_{0}^{1}(M) \;\text{such that}\; \int_{M} f |w + h|^{q} \, \,dv_{g} = \gamma \right\},
\end{equation*}
where $\gamma$ is a positive constant satisfying the condition
\begin{equation} \label{condition3-4}
\int_{M} f |h|^{2^{\sharp}} \, \,dv_{g} < \gamma.
\end{equation}
The following lemma, guarantees that the constraint set $\mathcal{H}_{\gamma,q}$ is non empty.

\begin{lemma} \label{lemma3-2}
If
\[
\int_{M} f |h|^{2^{\sharp}} \, \,dv_{g} < \gamma,
\]
then, the constraint set $\mathcal{H}_{\gamma, q}$ is non empty.
\end{lemma}
\begin{proof}
We set
\[
F_{q}(t) = \int_{M} f\, \left| t\psi_{1} + h \right|^{q} \, \,dv_{g},
\]
where $\psi_{1}$ is the eigenfunction corresponding to the first eigenvalue
$\lambda_{1}$ of the operator $-\operatorname{div}_{g}(a\nabla)$, that is
\[
\begin{cases}
-\operatorname{div}_{g}(a\nabla \psi_{1}) = \lambda_{1} \psi_{1} & \text{in } M, \\[0.4em]
\psi_{1} = 0 & \text{on } \partial M.
\end{cases}
\]
It is clear that $F_{q}$ is continuous for $q$ close to $2^{\sharp}$ and satisfies
\[
F_{q}(0) = \int_{M} f\, |h|^{q} \, \,dv_{g} < \gamma,
\quad
\lim_{t \to +\infty} F_{q}(t) = +\infty.
\]
By the intermediate value theorem, there exists $t_{\gamma,q} > 0$ such that
\[
F_{q}(t_{\gamma,q}) = \int_{M} f\, |t_{\gamma,q} \psi_{1} + h|^{q} \, \,dv_{g} = \gamma.
\]
Hence $t_{\gamma,q} \psi_{1} \in \mathcal{H}_{\gamma,q}$, which means that the constraint set is non empty, ending the proof of the lemma.
\end{proof}

By the following proposition, we show that the minimum $\mu_{\gamma,q}$ of the functional $I$, is attained by a smooth function belonging to $\mathcal{H}_{\gamma,q}$.

\begin{proposition} \label{proposition3-1}
Let $(M, g)$ be a smooth, compact Riemannian manifold of dimension $n > 3$ with smooth boundary. Assume that:
\begin{itemize}
    \item[(i)] the operator $-\operatorname{div}_{g}(a\nabla) + b$ is coercive;
    \item[(ii)] $\int_{M} f |h|^{2^{\sharp}} \,dv_{g} < \gamma$.
\end{itemize}
Then, for all $q \in (2, 2^{\sharp})$, there exist a real number $\lambda_{\gamma,q}$ and a smooth function $w_{\gamma,q} \in \mathcal{H}_{\gamma,q}$, solution to the problem \eqref{eq3-3}, which is a minimizer of $I$.
\end{proposition}
\begin{proof}
Firstly, we prove that the minimum $\mu_{\gamma,q}$ is finite.  
For $u\in H^{1}_{0}(M)$, H\"older’s inequality yields
\begin{equation}
    \bigg|\int_M bu^{2}\, dv_{g}\bigg|\leq V_{g}(M)^{1-\frac{2}{q}}\,
    \|b\|_{\infty}\,\Big(\int_{M}|u|^{q}\, dv_{g}\Big)^{\frac{2}{q}},\label{eq3-5}
\end{equation}
where $V_{g}(M)$ is the volume of $M$.  
Moreover, since $|u|^q \leq 2^{q-1}\big(|u+h|^q+|h|^q\big)$, for every $u\in H^{1}_{0}(M)$ we obtain
\begin{equation}
    \int_{M}|u|^{q}\, dv_{g}\leq 2^{q-1}(\min_{M} f)^{-1}\gamma. \label{eq3-6}
\end{equation}
Inserting \eqref{eq3-6} into \eqref{eq3-5} gives
\begin{equation}
     \bigg|\int_M bu^{2}\, dv_{g}\bigg|
     \leq V_{g}(M)^{1-\frac{2}{q}}\,\|b\|_{\infty}\,
     \Big(2^{q-1}(\min_{M}f)^{-1}\gamma\Big)^{\frac{2}{q}}.\label{eq3-7}
\end{equation}
Thus,
\begin{equation}
    I(u)\geq \int_{M}a|\nabla u|^{2}\,dv_{g}
    -C(\gamma,q),\label{eq3-8}
\end{equation}
where $ C(\gamma,q)=V_{g}(M)^{1-\frac{2}{q}}\,\|b\|_{\infty}\,
     \Big(2^{q-1}(\min_{M}f)^{-1}\gamma\Big)^{\frac{2}{q}}\geq 0$,
which shows that $\mu_{\gamma,q}>-\infty$.  

Let $(w_{i})_{i}$ be a minimizing sequence for $I$ on $\mathcal{H}_{\gamma,q}$.  
From \eqref{eq3-7} and the definition of $I$, for $i$ large enough we have
\begin{equation}
   \int_{M}a|\nabla w_i|^{2}\,dv_{g}
   \leq \mu_{\gamma,q}+1+ C(\gamma,q).\label{eq3-9}
\end{equation}
It follows that
\begin{equation}
    \int_{M}|\nabla w_i|^{2}\,dv_{g}
    \leq (\min_{M}|a|)^{-1}\big(\mu_{\gamma,q}+1+C(\gamma,q)\big). \label{eq3-10}
\end{equation}
On the other hand, \eqref{eq3-7} gives a uniform bound for $\|w_i\|_{2}^{2}$.  
Hence
\begin{equation}
\|w_i\|_{H^{1}_{0}(M)}^{2}\leq
(\min_{M}|a|)^{-1}\big(\mu_{\gamma,q}+1+C(\gamma,q)(\|b\|_{\infty}+\min_{M}|a|)\big). \label{eq3-11}
\end{equation}
Therefore $(w_i)_i$ is bounded in $H^{1}_{0}(M)$. By reflexivity, there exists a subsequence (still denoted $(w_i)$) such that
\begin{equation*}
    \begin{array}{ll}
(a) & w_{i}\rightharpoonup w_{\gamma,q}\quad \text{weakly in } H_{0}^{1}(M),\\[0.3em]
(b) & w_{i}\to w_{\gamma,q}\quad \text{strongly in } L^{s}(M)\text{ for all } s<2^{\sharp},\\[0.3em]
(c) & \|w_{\gamma,q}\|_{H^{1}_{0}(M)}\leq \liminf\limits_{i}\|w_{i}\|_{H^{1}_{0}(M)}.
\end{array}
\end{equation*}
It follows that
\[
I(w_{\gamma,q}) \leq \liminf_{i} I(w_i) = \mu_{\gamma,q}.
\]
Moreover, by the dominated convergence theorem,
\[
\int_{M}f|w_{\gamma,q}+h|^{q}\,dv_{g}
=\lim_{i}\int_{M}f|w_i+h|^{q}\,dv_{g}=\gamma.
\]
Thus, $w_{\gamma,q}\in\mathcal{H}_{\gamma,q}$ and $I(w_{\gamma,q})=\mu_{\gamma,q}$. By a well-known theorem of Lagrange, this gives the existence of a multiplier $\lambda_{\gamma,q}$ such that, for all $\varphi\in H^{1}_{0}(M)$,
\[
\int_{M}(a\nabla w_{\gamma,q}\cdot\nabla\varphi+bw_{\gamma,q}\varphi)\,dv_{g}
=\lambda_{\gamma,q}\int_{M}f|w_{\gamma,q}+h|^{q-2}(w_{\gamma,q}+h)\varphi\,dv_{g}.
\]
Equivalently, $w_{\gamma,q}$ is a weak solution of
\begin{equation}
\left\{
\begin{array}{lll}
-\operatorname{div}_{g}(a\nabla w_{\gamma,q})+bw_{\gamma,q}
=\lambda_{\gamma,q}\, f|w_{\gamma,q}+h|^{q-2}(w_{\gamma,q}+h) & \text{in }&M,\\[0.4em]
w_{\gamma,q}=0 & \text{on }&\partial M.
\end{array}
\right.\label{eq3-12}
\end{equation}
Finally, by standard elliptic regularity theory (see Gilbarg–Trudinger~\cite{Gilbard-Trudinger1983}), one has $w_{\gamma,q}\in C^{2,\alpha}(M)$ for some $\alpha\in(0,1)$.
\end{proof}

%\section{Critical solutions}

\section{Convergence to a critical solution}

\label{sec:sec4} In this section, we examine the behaviour of the
minimizing sequence $\left( w_{\gamma,q}\right) _{q}$ as the exponent $q$ tends to $2^{\sharp}=\frac{2n}{n-2}$. We first claim that the following holds. 

\begin{proposition}
\label{proposition4-1} 
Assume that 
\begin{equation*}
\int_{M} f |h|^{2^{\sharp}} \, dv_{g} < \gamma.
\end{equation*}
Then, the following assertions hold:
\begin{itemize}
\item[(i)] The sequence of Lagrange multipliers $(\lambda_{\gamma,q})_{q}$ is strictly positive and bounded. 
\item[(ii)] The sequence $\left(w_{\gamma,q}\right)_{q}$ is bounded in $H_{0}^{1}(M)$ and converges weakly to a smooth solution $u \in C^{2,\alpha}(M)$ of the critical equation \eqref{eq3-2}.
\end{itemize}
\end{proposition}

\begin{proof}
We recall that the condition 
\[\int_{M}f\left\vert h\right\vert
^{2^{\sharp}}\,dv_{g}<\gamma,\]
is valid for $q$ sufficiently close to $2^{\sharp}$.
To prove (i), we multiply the equation $\eqref{eq3-12}$ by $%
w_{\gamma,q}$ and we integrate over $M$ to get 

\begin{align}
0\leq\mu_{\gamma,q} & =\int_{M}\big(a|\nabla w_{\gamma,q}|^{2}+bw_{\gamma,q}^{2}\big) \,dv_{g} 
\notag \\[0.5em]
& =\lambda_{\gamma,q}\int_{M}f\left\vert w_{\gamma,q}+h\right\vert
^{q-2}\left( w_{\gamma,q}+h\right) w_{\gamma,q}\,dv_{g}\notag \\[0.5em]
& =\lambda_{\gamma,q}\int_{M}f\left\vert w_{\gamma,q}+h\right\vert
^{q-2}\left( w_{\gamma,q}+h\right) \left( w_{\gamma,q}+h-h\right) \,dv_{g} 
\notag \\[0.5em]
& =\lambda_{\gamma,q}\left( \gamma-\int_{M}f\left\vert w_{\gamma
,q}+h\right\vert ^{q-2}\left( w_{\gamma,q}+h\right) h \,dv_{g}\right).
\label{eq4-1}
\end{align}
From the assumption $\int_{M} f\, |h|^{2^{\sharp}} \, dv_{g} < \gamma ,$ and by H\"older's inequality, we obtain
\begin{equation}
\int_{M}f\left\vert w_{\gamma,q}+h\right\vert ^{q-2}\left(
w_{\gamma,q}+h\right) h\,dv_{g}\leq\gamma^{1-\frac{1}{q}}\left( \int
\nolimits_{M}f\left\vert h\right\vert ^{q}\right) ^{\frac{1}{q}}<\gamma. \label{eq4-2}
\end{equation}
It follows from \eqref{eq4-1} and \eqref{eq4-2} that $\lambda_{\gamma,q}\geq0$. Moreover, if $\lambda_{\gamma,q}=0$, we get $w_{\gamma,q}=0\in \mathcal{H}_{\gamma,q}$, that is,
\begin{equation*}
\int_{M} f |h|^q \, dv_{g} = \gamma,
\end{equation*}
which contradicts the assumption 
\begin{equation}
\int_{M} f |h|^{q} \, dv_{g} < \gamma ,\label{eq4-3}
\end{equation}
and therefore $\lambda_{\gamma,q} \neq 0$. In order to prove that the sequence $\left(\lambda_{\gamma,q}\right)_{q}$ is
bounded, we denote by $\psi_{1}$ the eigenfunction
corresponding to the first eigenvalue $\lambda_{1}$ of the operator 
$-\operatorname{div}_{g}(a\nabla)$. It satisfies
\begin{equation*}
\left\{ 
\begin{array}{lll}
-\operatorname{div}_{g}(a\nabla \psi_{1})=\lambda_{1}\psi_{1}, & \text{in }& M, \\[0.5em] 
\psi_{1}=0, & \text{on } &\partial M, \\[0.5em] 
\displaystyle \int_{M}\psi_{1}^{2}\, dv_{g} = 1. & &
\end{array}
\right. 
\end{equation*}
We set 
\begin{equation*}
F(t,q)=\int_{M} f |t\psi_{1}+h|^{q}\, dv_{g}.
\end{equation*}
According to Lemma~\ref{lemma3-2}, there exists $t_{\gamma,q}>0$ such that 
$t_{\gamma,q}\psi_{1}\in \mathcal{H}_{\gamma,q}$. That is 
\begin{equation*}
F(t_{\gamma,q},q)=\int_{M} f |t_{\gamma,q}\psi_{1}+h|^{q}\, dv_{g}=\gamma.
\end{equation*}
To prove that 
\[
\frac{\partial F}{\partial t}\left( t_{\gamma,q},q \right) \neq 0,
\]
we argue by contradiction. Assume that 
\[
\frac{\partial F}{\partial t}\left( t_{\gamma,q},q \right) = 0.
\] 
We can write 
\begin{align*}
t_{\gamma,q}\,\frac{\partial F}{\partial t}\left( t_{\gamma,q},q\right) 
&= q t_{\gamma,q} \int_{M} f \left| t_{\gamma,q}\psi_{1}+h \right|^{q-2}
       \left( t_{\gamma,q}\psi_{1}+h \right)\psi_{1}\, dv_{g} \\[0.5em]
&= q \int_{M} f \left| t_{\gamma,q}\psi_{1}+h \right|^{q-2}
       \left( t_{\gamma,q}\psi_{1}+h \right)
       \left( t_{\gamma,q}\psi_{1}+h - h \right)\, dv_{g} \\[0.5em]
&= q \left( \gamma - \int_{M} f \left| t_{\gamma,q}\psi_{1}+h \right|^{q-2}
       \left( t_{\gamma,q}\psi_{1}+h \right) h \, dv_{g} \right) \\[0.5em]
&= 0.
\end{align*}
Hence 
\begin{equation}
\gamma=\int_{M}f\left\vert t_{\gamma,q}\psi_{1}+h\right\vert ^{q-2}\left(
t_{\gamma,q}\psi_{1}+h\right) h\,dv_{g}\text{.}  \label{eq4-4}
\end{equation}
On the other hand, from H\"older's inequality and the assumption that 
\begin{equation*}
\int_{M} f\left\vert h\right\vert ^{q}\,dv_{g}<\gamma,
\end{equation*}
we deduce that 
\begin{align*}
\int_{M}f\left\vert t_{\gamma,q}\psi_{1}+h \right\vert  ^{q-2} \left(
t_{\gamma,q}\psi_{1} +h\right) h\,dv_{g}
&\leq \left( \int_{M}f\left\vert t_{\gamma,q}\psi_{1}+h\right\vert
^{q}\,dv_{g}\right) ^{1-\frac{1}{q}} \left( \int_{M}f\left\vert h\right\vert
^{q}\,dv_{g}\right) ^{\frac{1}{q}}\\[0.5em]
&< \gamma,
\end{align*}
which contradicts the relation \eqref{eq4-4}. Therefore
\[
\frac{\partial F}{\partial t}\left( t_{\gamma,q},q\right) \neq 0.
\]
By the implicit function theorem, it follows that the function $t_{\gamma,q}$ 
depends continuously on $q$. Consequently, there exists a constant 
$C(\gamma)>0$, independent of $q$, such that 
\begin{align}
\int_{M}\bigl(a|\nabla w_{\gamma,q}|^2 + b w_{\gamma,q}^2\bigr)\, dv_{g}
   &\leq I\bigl(t_{\gamma,q}\psi_{1}\bigr)  \notag \\[0.5em]
   &= t_{\gamma,q}^{2}\, I(\psi_{1}) \notag \\[0.5em]
   &\leq C(\gamma)\big(\lambda_1+\Vert b\Vert_\infty\big).
\label{eq4-5}
\end{align}
According to \eqref{eq4-2} and the estimate above, 
we obtain 
\begin{align*}
0 < \lambda_{\gamma,q} 
&= \frac{\displaystyle \int_{M}\bigl(a|\nabla w_{\gamma,q}|^2 + b w_{\gamma,q}^2\bigr)\, dv_{g}}
        {\displaystyle \int_{M} f \left| w_{\gamma,q}+h \right|^{q-2} 
        \left( w_{\gamma,q}+h \right) w_{\gamma,q}\, dv_{g}} \\[1em]
&= \frac{\displaystyle \int_{M}\bigl(a|\nabla w_{\gamma,q}|^2 + b w_{\gamma,q}^2\bigr)\, dv_{g}}
        {\gamma - \displaystyle \int_{M} f \left| w_{\gamma,q}+h \right|^{q-2}
        \left( w_{\gamma,q}+h \right) h \, dv_{g}} \\[1em]
&\leq \frac{I\!\left( t_{\gamma,q}\psi_{1}\right)}
        {\gamma - \gamma^{1-\frac{1}{q}}
        \left( \int_{M} f |h|^{q}\, dv_{g} \right)^{\frac{1}{q}}} \\[1em]
&\leq \frac{C(\gamma)\big(\lambda_1+\Vert b\Vert_\infty\big)}{\gamma - \gamma^{1-\frac{1}{q}}
        \left( \int_{M} f |h|^{q}\, dv_{g} \right)^{\frac{1}{q}}} \\[1em]
&\leq C'(\gamma,h).
\end{align*}
Hence, the sequence $\bigl(\lambda_{\gamma,q}\bigr)_{q}$ remains bounded 
as $q\to 2^{\sharp}$. Therefore, there exists a subsequence, still denoted 
$\big(\lambda_{\gamma,q}\big)_{q}$, such that $\lim_{q\to2^{\sharp}}\lambda_{\gamma,q}= \lambda$. Let us now prove the second statement of the proposition. 
By the estimate \eqref{eq4-5} and the coercivity of the operator 
$-\operatorname{div}_{g}(a\nabla) + b$, 
we deduce that the sequence $\{w_{\gamma,q}\}_{q}$ is bounded in the reflexive space $H^{2}_{3,0}(M)$. 
Therefore, there exists $w \in H^{1}_{0}(M)$ such that, up to a subsequence, 
\begin{align*}
(a)\;& w_{\gamma,q} \rightharpoonup w 
&& \text{weakly in } H^{1}_{0}(M), \\[0.5em]
(b)\;& w_{\gamma,q} \to w 
&& \text{strongly in } L^{s}(M) 
\;\; \text{for all } s<2^{\sharp}, \\[0.5em]
(c)\;& w_{\gamma,q} \to w 
&& \text{a.e.\ in } M. 
\end{align*}
Passing to the limit in the equation in \eqref{eq3-12} as $q\to 2^{\sharp}$, we deduce that $w$ is a weak
solution for the critical equation \eqref{eq3-2}. Consequently, setting $u := w+h$, 
we obtain that $u$ is a weak solution of equation \eqref{eq1-4}. 
By regularity theory from Gilbarg-Trudinger~\cite{Gilbard-Trudinger1983}, 
it follows that $u \in C^{2,\alpha}(M)$ for any $0<\alpha<1$. 
This completes the proof.
\end{proof}
\subsection{Condition for a nontrivial solution}
\begin{remark}
\label{remark3-1} 
It is clear that if the boundary data $\phi\not \equiv 0$, then
the solution of \eqref{eq1-4} is non-identically zero (i.e., $u\not \equiv 0$). 
Otherwise, if $\phi=0$, then $h=0$, so that $u=w$. In this case, under an additional
assumption, we prove that $w$ is a nontrivial solution of 
\begin{equation}
\left\{ 
\begin{array}{lll}
-\operatorname{div}_{g}\!\left(a\nabla w\right)+bw=\lambda f\vert w\vert ^{2^{\sharp}-2}w & \text{in }&M,
\\[0.5em] 
w=0 & \text{on }&\partial M.
\end{array}
\right.  \label{eq4-6}
\end{equation}
\end{remark}
We want to formulate the condition that ensure the
existence of a nontrivial solution to the problem \eqref{eq4-6}. For any $u\in H^{1}_{0}(M)\setminus \{0\}$, we define  
\[
I(u)=\int_{M}\left(a|\nabla u|^2+bu^2\right)\,dv_{g}.
\]
The associated Rayleigh quotient is given by  
\begin{equation*}
Q(u)=\frac{I(u)}{\left( \displaystyle\int_{M}f|u|^{2^{\sharp}}\,dv_{g}\right)^{\frac{2}{2^{\sharp}}}}.
\end{equation*}
The corresponding minimization problem is  
\begin{equation}
\mu = \inf_{u\in H^{1}_{0}(M)\setminus \{0\}} Q(u).
\label{eq4-7}
\end{equation}
Since both terms in $Q$ are one-homogeneous, $Q$ is invariant under
multiplication of $u$ by a non-zero constant. Thus, the minimization problem \eqref{eq4-7} can be reformulated as  
\begin{equation}
\mu = \inf_{u\in \mathcal{H}_{\gamma}} I(u),
\label{eq4-8}
\end{equation}
where  
\[
\mathcal{H}_{\gamma}=\left\{ u\in H^{1}_{0}(M)\;:\;\int_{M}f|u|^{2^{\sharp}}\,dv_{g}=\gamma \right\},
\]
and $\gamma>0$ is chosen as in \eqref{condition3-4}. We have the following lemma.

\begin{lemma}
The subcritical minimum $\mu_{\gamma ,q}$ converges 
to the critical minimum $\mu$, given by \eqref{eq4-8}, as $q$ tends to $2^{\sharp}$.
\end{lemma}

\begin{proof}
The result follows essentially from the fact that the sequence 
$(\lambda_{\gamma ,q})_{q}$ of Lagrange multipliers associated with the subcritical equations converges to $\lambda$, the Lagrange multiplier of the 
critical equation. Indeed, by formula \eqref{eq4-1}, we have 
\begin{equation*}
\mu_{\gamma ,q}
 = \lambda_{\gamma ,q}\left( \gamma - \int_{M} f \vert w_{\gamma ,q}+h\vert^{\,q-2} (w_{\gamma ,q}+h) h \, dv_{g}\right).
\end{equation*}
Since 
\[
\lim_{q\to 2^{\sharp}} \lambda_{\gamma ,q} = \lambda,
\]
and \[\lim_{q\to 2^{\sharp}} \int_{M} f \vert w_{\gamma ,q}+h\vert^{\,q-2} (w_{\gamma ,q}+h) h \, dv_{g} 
= \int_{M} f \vert w+h\vert^{2^{\sharp}-2} (w+h) h \, dv_{g},\]
we infer that
\begin{equation*}
\lim_{q\to 2^{\sharp}} \mu_{\gamma ,q} 
= \lambda \left( \gamma - \int_{M} f \vert w+h\vert^{2^{\sharp}-2} (w+h) h \, dv_{g}\right).
\end{equation*}
Moreover, from the proof of Proposition~\ref{proposition4-1}, it has been shown that 
$\mu$ is the minimum of the functional 
\[
I(w) = \int_{M}\left(a\vert\nabla w\vert^2 + b w^2\right)\, dv_{g}
\quad \text{on } \mathcal{H}_{\gamma}.
\]
Therefore,
\begin{eqnarray*}
\mu  
&=& \int_{M}\left(a\vert\nabla w\vert^2 + b w^2\right)\, dv_{g} \\
&=& \lambda \int_{M} f \vert w+h\vert^{2^{\sharp}-2} (w+h) w \, dv_{g} \\
&=& \lambda \left( \gamma - \int_{M} f \vert w+h\vert^{2^{\sharp}-2} (w+h) h \, dv_{g}\right) \\
&=& \lim_{q\to 2^{\sharp}} \mu_{\gamma ,q}.
\end{eqnarray*}
\end{proof}

\begin{proposition}
\label{proposition3-4-2} 
Let $w$ be the weak limit of the minimizing sequence $\left(w_{\gamma ,q}\right)_{q}$ in $H^{1}_{0}\left(M\right)$, and assume that 
\begin{equation}
K_{0}\, \big(\min_{M}a\big)^{-1}\, \big(\displaystyle\max_{M} f\big)^{\tfrac{2}{2^{\sharp}}}\,\gamma^{-\tfrac{2}{2^{\sharp}}} \,\mu < 1,
\label{eq4-9}
\end{equation}
where $\mu=\lim_{q\rightarrow 2^{\sharp}}\mu_{\gamma,q}$ and $K_{0}$ denotes the best Sobolev constant.  
Then $w$ is a nontrivial solution of the problem \eqref{eq4-6}.
\end{proposition}

\begin{proof}
We argue by contradiction. Assume that $w\equiv 0$, then, $\left(w_{\gamma,q}\right)_{q}$ converges weakly to $0$ in $H^{1}_{0}(M)$. Since $H^{1}_{0}(M)$ is compactly embedded in $L^{2}(M)$, it follows that $\left(w_{\gamma,q}\right)_{q}$ converges strongly to $0$ in $L^{2}(M)$ as $q\to 2^{\sharp}$. We have
\begin{equation*}
\gamma^{\frac{2}{q}}
= \left(\int_{M} f|w_{\gamma,q}|^{q}\,dv_{g}\right)^{\frac{2}{q}}
\leq \big(\max_{M}f\big)^{\frac{2}{q}}
V_{g}(M)^{\frac{2}{q}-\frac{2}{2^{\sharp}}}
\left(\int_{M}|w_{\gamma,q}|^{2^{\sharp}}\,dv_{g}\right)^{\frac{2}{2^{\sharp}}},
\end{equation*}
where $V_{g}(M)$ denotes the volume of $M$. By applying the Sobolev inequality \eqref{eq2-3} (see Lemma \ref{lemme2-1}), we deduce that for every $\varepsilon>0$, there exists $B_{\varepsilon}>0$ such that
\begin{equation}
\gamma^{\frac{2}{q}}\leq 
\big(\max_{M}f\big)^{\frac{2}{q}}
V_{g}(M)^{\frac{2}{q}-\frac{2}{2^{\sharp}}}
\left((K_{0}+\varepsilon)\|\nabla w_{\gamma,q}\|_{2}^{2}
+ B_{\varepsilon}\|w_{\gamma,q}\|_{2}^{2}\right).
\label{eq4-10}
\end{equation}
Since $I(w_{\gamma,q})=\mu_{\gamma,q}$, we have
\begin{equation}
\int_{M}a|\nabla w_{\gamma,q}|^{2}\,dv_{g}
= \mu_{\gamma,q}-\int_{M}b w_{\gamma,q}^{2}\,dv_{g}.
\label{eq4-11}
\end{equation}
Hence
\begin{align}
\int_{M}|\nabla w_{\gamma,q}|^{2}\,dv_{g}
&\leq (\min_{M}a)^{-1}\int_{M}a|\nabla w_{\gamma,q}|^{2}\,dv_{g} \notag\\[0.5em]
&\leq (\min_{M}a)^{-1}\left(\mu_{\gamma,q}-\int_{M}b w_{\gamma,q}^{2}\,dv_{g}\right) \notag\\[0.5em]
&\leq (\min_{M}a)^{-1}\Big(\mu_{\gamma,q}+\|b\|_{\infty}\|w_{\gamma,q}\|_{2}^{2}\Big).
\label{eq4-12}
\end{align}
Combining \eqref{eq4-10} and \eqref{eq4-12}, we obtain
\begin{align}
\gamma^{\frac{2}{q}}
&\leq \big(\max_{M}f\big)^{\frac{2}{q}}
V_{g}(M)^{\frac{2}{q}-\frac{2}{2^{\sharp}}}
\Big\{ (K_{0}+\varepsilon)(\min_{M}a)^{-1}\big(\mu_{\gamma,q}+\|b\|_{\infty}\|w_{\gamma,q}\|_{2}^{2}\big) \notag\\
&\quad + B_{\varepsilon}\|w_{\gamma,q}\|_{2}^{2}\Big\}.
\label{eq4-13}
\end{align}
Passing to the limit as $q\to 2^{\sharp}$, and noting that 
\[
\|w_{\gamma,q}\|_{2}^{2}=o(1) 
\qquad \text{and} \qquad 
V_{g}(M)^{\frac{2}{q}-\frac{2}{2^{\sharp}}}=1+o(1),
\]
we conclude that
\begin{equation*}
K_{0}\, \big(\min_{M}a\big)^{-1}\, \big(\displaystyle\max_{M} f\big)^{\tfrac{2}{2^{\sharp}}}\,\gamma^{-\tfrac{2}{2^{\sharp}}} \,\mu\geq 1,
\end{equation*}
which contradicts the assumption \eqref{eq4-9}. Therefore, $w\not\equiv 0$. 
\end{proof}
\section{Test functions and geometric conditions}\label{sec:sec5}

This section is devoted to evaluate the condition \eqref{eq4-9} by means of test functions. For this purpose, we consider a system of normal
coordinate $\left( x^{1},\, x^{2},\,\ldots,\,x^{n}\right) $ centred
at a point $x_{0}$ where the function $f$ is maximal (i.e. $f(
x_{0}) =\max_{M}f$). Let $B\left( x_{0},\delta\right) $ be the geodesic ball of radius $\delta$
centred at $x_{0}$ such that the injectivity radius $i_{g}(M) $ of the
metric g at the point $x_{0}$ is greater than $2\delta$. Let $\eta$ be a smooth cut-off function defined by 
\begin{equation*}
\eta(x) =\left\{ 
\begin{array}{lll}
1 & \text{if }&x\in B\left( x_{0},\delta\right), \\[0.5em] 
0 & \text{if }&x\in M-B\left( x_{0},2\delta\right).%
\end{array}
\right. 
\end{equation*}
For $\varepsilon>0$, we define the radial function $u_{\varepsilon}\in
C^{\infty}(M) $ by 
\begin{equation*}
u_{\varepsilon}(x)=\eta(x)\,v_{\varepsilon}(x), 
\end{equation*}
where 
\begin{equation*}
v_{\varepsilon}(x)=\left(\frac{2\varepsilon}{\varepsilon^{2}+r^{2}}\right)^{\frac{n-2}{2}},
\end{equation*}
where $r=d_{g}(x,x_{0})$. Our aim is to estimate 
\begin{equation}
    \mu_{\varepsilon} =\int_{M}\left(a|\nabla u_{\varepsilon }|^{2}+bu_{\varepsilon }^{2}\right)\, dv_{g} , \label{eq5-1}
\end{equation}
and
\begin{equation}
 \\
\gamma_{\varepsilon} =\int_{M}f\left\vert u_{\varepsilon }\right\vert
^{2^{\sharp}}\,dv_{g},  \label{eq5-2}
\end{equation}
by the radial functions $u_{\varepsilon}$, which will be used to prove that
the quotient 
\begin{equation}
Q_{\varepsilon}=K_{0}\, \big(a(x_0))^{-1}\,\big(f(x_0)\big)^{\frac{2}{2^{\sharp}}
}\,\gamma_{\varepsilon}^{-\frac{2}{2^{\sharp}}}\,\mu_{\varepsilon},
\label{eq5-3}
\end{equation}
is strictly less than $1$, when $\varepsilon$ is small enough. We recall the following elementary identities (see \cite{Aubin1976}), which will be used later in the computations.
\begin{lemma}\cite{Aubin1976}\label{lemma5-1}
Let $p,q$ be two positive real numbers. 
For $p-q>1$, set
\[
I_{p}^{q}=\int_{0}^{\infty}\frac{t^{q}}{(1+t)^{p}}\,dt.
\]
Then, the following hold:
\begin{enumerate}
\item $I_{p+1}^{q}=\frac{p-q-1}{p}I_{p}^{q}$,  \quad $I_{p+1}^{q+1}=\frac{q+1}{p-q-1}I_{p+1}^{q}$.
\item If $\delta \in \mathbb{R}^{+}$, the limit 
\[
\lim_{\varepsilon \to 0^{+}} \left( \int_{0}^{\delta}\frac{t^{q}}{(t+\varepsilon)^{p}}\,dt
-\frac{I^{q}_{p}}{\varepsilon^{p-q-1}} \right)
\]
is finite whenever $p-q-1>0$.

\item \label{3} If $\delta \in \mathbb{R}^{+}$, the limit
\[
\lim_{\varepsilon \to 0^{+}} \left( \int_{0}^{\delta}\frac{t^{q}}{(t+\varepsilon)^{p}}\,dt
- \log\!\frac{1}{\varepsilon}\right)
\]
is finite whenever $p-q-1=0$.
\end{enumerate}
\end{lemma}

\begin{lemma}\cite{Aubin1976} \label{lemma5-2}
Let $(M,g)$ be a Riemannian manifold and 
$(x^{1},\ldots,x^{n})$ be a system of normal coordinates centred at 
$x_{0}\in M$. Denote by $S(r)$ the geodesic sphere of radius $r$ centred at $x_{0}$ such that $r<d$ ($d$ stands for the injectivity radius of the manifold $M$). Define
\[
G(r)=\frac{1}{\omega_{n-1}}\int_{S(r)}\sqrt{|g|}\,d\sigma,
\]
where $d\sigma$ and $\omega_{n-1}$  denote the volume element and the volume of the $(n-1)$-Euclidean unit sphere $\mathbb{S}^{n-1}$, respectively, and $|g|$ is the determinant of the Riemannian metric. We have
\begin{enumerate}
\item The Taylor's expansion of $\sqrt{|g|}$ in a neighborhood of $r=0$ is
given by
\[
\sqrt{|g|}=1-\tfrac{1}{6}R_{ij}x^{i}x^{j}+o(r^{2}),
\]
where $R_{ij}$ are the components of the Ricci curvature.
\item The Taylor's expansion of $G(r)$ in a neighborhood of $r=0$ is
given by
\[
G(r)=1-\frac{R}{6n}r^{2}+o(r^{2}).
\]
\end{enumerate}
\end{lemma}

For the Taylor's expansion of $Q_{\protect\varepsilon}$, we distinguish two cases, the case when $n>4 $ and the case when $n=4$.

\subsection{Taylor's expansion of $Q_{\protect\varepsilon}$ for $n>4 $}

Firstly, We estimate the different terms of $\mu_{\varepsilon}$ in \eqref{eq5-1}
separately. To do so, we use similar arguments as in \cite{Bekiri-Benalili2018,Holcman2001}. We begin with the leading term in $\mu_{\varepsilon}$. We have 
\begin{equation}
\int_{M}a|\nabla
u_{\varepsilon}|^{2}\,dv_{g}=\int_{B\left(x_{0},\delta\right)}a|\nabla
v_{\varepsilon}|^{2}\,dv_{g}+\int_{B\left(x_{0},2\delta\right)\setminus B\left(x_{0},\delta\right)}a|\nabla
(\eta v_{\varepsilon})|^{2}\,dv_{g}.  \label{eq5-4}
\end{equation}
Moreover, we observe that 
\begin{equation*}
\nabla v_{\varepsilon}=\frac{\partial v_{\varepsilon} }{\partial r}=-(n-2)\frac{rv_{\varepsilon}}{
\varepsilon^{2}+r^{2}}.
\end{equation*}
Using polar coordinates and the following expansion (see lemma~\ref{lemma5-2})
\[\sqrt{|g|}=1-\tfrac{1}{6}R_{ij}x^{i}x^{j}+o(r^{2}),\] 
we obtain 
\begin{align*}
\int_{B\left(x_{0},\delta\right)}a|\nabla
v_{\varepsilon}|^{2}\,dv_{g}&= \int_{0}^{\delta}r^{n-1}\big(\frac{\partial v_{\varepsilon} }{\partial r}\big)^2\bigg(\int_{S(r)}a\sqrt{|g|}\,d\Omega\bigg)\\[0.5em]
&=\omega_{n-1}(n-2)^{2}\\[0.5em]
&\times\int_{0}^{\delta}\frac{r^{n+1}v_{\varepsilon}^{2}}{\left(
\varepsilon^{2}+r^{2}\right)^{2}}\left(a(x_{0})-\left( \frac{\Delta_{g}a(x_{0})}{2n}+\frac{a(x_{0})}{6n}R_{g}(x_{0})\right)r^{2}+o(r^{2})  \right) dr
\end{align*}
where $\omega_{n-1}$ denotes the volume of the $(n-1)$-Euclidean unit sphere $
\mathbb{S}^{n-1}$. By using the change of variables $r=\varepsilon\sqrt{t},\; dr=\frac{\varepsilon}{2\sqrt{t}}$, and letting $\varepsilon\rightarrow 0$, we obtain 
\begin{align*}
    \int_{B\left(x_{0},\delta\right)}a|\nabla
v_{\varepsilon}|^{2}\,dv_{g}=&\omega_{n-1}(n-2)^{2} 2^{n-3}\\[0.5em]
& \times\left(a(x_{0})I^{\frac{n}{2}}_{n}-I^{\frac{n}{2}+1}_{n}\left( \frac{\Delta_{g}a(x_{0})}{2n}+\frac{a(x_{0})}{6n}R_{g}(x_{0})\right)\varepsilon^{2}+o(\varepsilon^{2})  \right).
\end{align*}
From the relation 
\begin{equation*}
I^{\frac{n}{2}+1}_{n}=\frac{n+2}{n-4}I^{\frac{n}{2}}_{n},
\end{equation*}
we deduce that
\begin{align}
\int_{B\left(x_{0},\delta\right)}a|\nabla
v_{\varepsilon}|^{2}\,dv_{g}=&\omega_{n-1}(n-2)^{2} 2^{n-3}I^{\frac{n}{2}}_{n}a(x_{0})\nonumber\\[0.5em]
&\times\bigg(1-\frac{n+2}{2n(n-4)}\left( \frac{\Delta_{g}a(x_{0})}{a(x_{0})}+\frac{R_{g}(x_{0})}{3}\right)\varepsilon^{2}+o(\varepsilon^{2})  \bigg).
\label{eq5-5}
\end{align}
Similarly, as $\varepsilon\rightarrow 0$, and $n>4$, we obtain
\begin{equation}
\int_{B\left(x_{0},2\delta\right)-B\left(x_{0},\delta\right)}\vert
\nabla (\eta v_{\varepsilon})\vert^{2}\,dv_{g}=O(\varepsilon^{n-2}).
\label{eq5-6}
\end{equation}
Inserting \eqref{eq5-5}, \eqref{eq5-6} into \eqref{eq5-4}, and using the relations 
\[\omega_{n}=2^{n-1}I^{\frac{n}{2}-1}_{n}\omega_{n-1}\; \text{ and }\; I^{\frac{n}{2}}_{n}=\frac{n}{n-2}I^{\frac{n}{2}-1}_{n},\]
we deduce that 
\begin{equation}
\int_{M}a|\nabla
u_{\varepsilon}|^{2}\,dv_{g}=\frac{n(n-2)\omega_{n}}{4}a(x_{0})\bigg(1-\frac{n+2}{2n(n-4)}\left( \frac{\Delta_{g}a(x_{0})}{a(x_{0})}+\frac{R_{g}(x_{0})}{3}\right)\varepsilon^{2}+o(\varepsilon^{2})  \bigg).  \label{eq5-7}
\end{equation}
For the remaining term in $\mu_\varepsilon$, we have
\begin{equation}
\int_{M}bu_{\varepsilon}^{2}\,dv_{g}=\int_{B\left(x_{0},\delta
\right)}bv_{\varepsilon}^{2}\,dv_{g}+\int_{B\left(x_{0},2\delta\right)\setminus B\left(x_{0},\delta\right)}b(\eta v_{\varepsilon})^{2}\,dv_{g}.
\label{eq5-8}
\end{equation}
Using polar coordinates together with the expansion 
\[\sqrt{|g|}=1-\tfrac{1}{6}R_{ij}x^{i}x^{j}+o(r^{2}),\] 
we obtain 
\begin{samepage}
\begin{align*}
\int_{B\left(x_{0},\delta\right)}bv_{\varepsilon}^{2}\,dv_{g}&= \int_{0}^{\delta}r^{n-1}v_{\varepsilon}^2\bigg(\int_{S(r)}b\sqrt{|g|}\,d\Omega\bigg)\\[0.5em]
&=\omega_{n-1} \int_{0}^{\delta}r^{n-1}\bigg(\frac{2\varepsilon}{
\varepsilon^{2}+r^{2}}\bigg)^{n-2}\\[0.5em]
&\quad\times\left(b(x_{0})-\left( \frac{\Delta_{g}b(x_{0})}{2n}+\frac{b(x_{0})R_{g}}{6n}\right)r^{2}+o(r^{2})  \right) dr.
\end{align*}
\end{samepage}
Applying the change of variables $r=\varepsilon\sqrt{t},\; dr=\frac{\varepsilon}{2\sqrt{t}}$, and letting $\varepsilon\rightarrow 0$, we get
\begin{align*}
\int_{B\left(x_{0},\delta\right)}bv_{\varepsilon}^{2}\,dv_{g}
&=2^{n-3}\omega_{n-1} \left(b(x_{0})I^{\frac{n}{2}-1}_{n-2}\varepsilon^{2}+o(\varepsilon^{2})  \right),
\end{align*}
where
\[I^{\frac{n}{2}-1}_{n-2}=4\frac{(n-2)(n-1)}{n(n-4)}I^{\frac{n}{2}}_{n}.\]
Therefore,
\begin{equation}
    \int_{B\left(x_{0},\delta\right)}bv_{\varepsilon}^{2}\,dv_{g}
=2^{n-3}I^{\frac{n}{2}}_{n}\omega_{n-1} \left(4b(x_{0})\frac{(n-2)(n-1)}{n(n-4)}\varepsilon^{2}+o(\varepsilon^{2})  \right).\label{eq5-9}
\end{equation}
Similarly, one easily checks that 
\begin{equation}
\int_{B\left(x_{0},2\delta\right)\setminus B\left(x_{0},\delta\right)}b(\eta
v_{\varepsilon})^{2}\,dv_{g}=O(\varepsilon^{n-2}).  \label{eq5-10}
\end{equation}
Inserting \eqref{eq5-9} and \eqref{eq5-10}, into \eqref{eq5-8} and using the relations
\[\omega_{n}=2^{n-1}I^{\frac{n}{2}-1}_{n}\omega_{n-1}\; \text{ and }\; I^{\frac{n}{2}}_{n}=\frac{n}{n-2}I^{\frac{n}{2}-1}_{n},\]
we finally deduce that
\begin{equation}
    \int_{M}bu_{\varepsilon}^{2}\,dv_{g}
=\omega_{n} \left(b(x_{0})\frac{(n-1)}{(n-4)}\varepsilon^{2}+o(\varepsilon^{2})  \right).\label{eq5-11}
\end{equation}
Plugging \eqref{eq5-7} and \eqref{eq5-11} into \eqref{eq5-1}, we obtain 
\begin{align}
    \mu_{\varepsilon}&=\frac{n(n-2)\omega_{n}}{4}\, a(x_{0})\notag\\[0.5em]
    &\times\bigg\lbrace1+\frac{1}{2n(n-2)(n-4)}\bigg(8(n-1)\frac{b(x_{0})}{a(x_{0})}-(n^2-4)\big( \frac{\Delta_{g}a(x_{0})}{a(x_{0})}+\frac{R_{g}(x_{0})}{3}\big)\bigg)\varepsilon^2+o(\varepsilon^2)\bigg\rbrace.
\label{eq5-12}
\end{align}
Now, in order to estimate $\gamma_{\varepsilon}$, we split it into two parts. That is 
\begin{equation}
\gamma_{\varepsilon}=\int_{B(x_{0},\delta)}f\left\vert v_{\varepsilon
}\right\vert
^{2^{\sharp}}\,dv_{g}+\int_{B(x_{0},2\delta)\setminus B(x_{0},\delta)}f\left\vert \eta
v_{\varepsilon }\right\vert ^{2^{\sharp}}\,dv_{g}.  \label{eq5-13}
\end{equation}
Using polar coordinates together with the expansion 
\[\sqrt{|g|}=1-\tfrac{1}{6}R_{ij}x^{i}x^{j}+o(r^{2}),\]
we obtain
\begin{align*}
\int_{B\left(x_{0},\delta\right)}f|v_{\varepsilon}|^{2^{\sharp}}\,dv_{g}&= \int_{0}^{\delta}r^{n-1}|v_{\varepsilon}|^{2^{\sharp}}\bigg(\int_{S(r)}f\sqrt{|g|}\,d\Omega\bigg)\\[0.5em]
&=\omega_{n-1} \int_{0}^{\delta}r^{n-1}\bigg(\frac{2\varepsilon}{
\varepsilon^{2}+r^{2}}\bigg)^{n}\\[0.5em]
&\quad\times\left(f(x_{0})-\left( \frac{\Delta_{g}f(x_{0})}{2n}+\frac{f(x_{0})}{6n}R_{g}\right)r^{2}+o(r^{2})  \right) dr.
\end{align*}
Once again, applying the change of variables $r=\varepsilon\sqrt{t},\; dr=\frac{\varepsilon}{2\sqrt{t}}$, and letting $\varepsilon\rightarrow 0$, we obtain 
\begin{align*}
\int_{B\left(x_{0},\delta\right)}f|v_{\varepsilon}|^{2^{\sharp}}\,dv_{g}=&2^{n-1}\omega_{n-1}\\[0.5em]
& \times\left(I^{\frac{n}{2}-1}_{n}f(x_{0})-I^{\frac{n}{2}}_{n}\left( \frac{\Delta_{g}f(x_{0})}{2n}+\frac{f(x_{0})}{6n}R_{g}(x_{0})\right)\varepsilon^{2}+o(\varepsilon^{2})  \right).
\end{align*}
By using the relations
\[I^{\frac{n}{2}}_{n}=\frac{n}{n-2}I^{\frac{n}{2}-1}_{n} \; \text{ and }\;\omega_{n}=2^{n-1}I^{\frac{n}{2}-1}_{n}\omega_{n-1},\]
we deduce that
\begin{equation}
\int_{B\left(x_{0},\delta\right)}f|v_{\varepsilon}|^{2^{\sharp}}\,dv_{g}=\omega_{n}f(x_{0})\left(1-\frac{1}{2(n-2)}\left( \frac{\Delta_{g}f(x_{0})}{f(x_{0})}+\frac{R_{g}(x_{0})}{3}\right)\varepsilon^{2}+o(\varepsilon^{2})  \right).\label{eq5-14}
\end{equation}
One can easily check that
\begin{equation}
\int_{B(x_{0},2\delta)\setminus B(x_{0},\delta)}f| \eta v_{\varepsilon
}| ^{2^{\sharp}}\,dv_{g}=O(\varepsilon^{n-2}).\label{eq5-15}
\end{equation}
Inserting \eqref{eq5-14} and \eqref{eq5-15} into \eqref{eq5-13}, we obtain 
\begin{equation}
\gamma_{\varepsilon}=\omega_{n}f(x_{0})\left(1-\frac{1}{2(n-2)}\left( \frac{\Delta_{g}f(x_{0})}{f(x_{0})}+\frac{R_{g}(x_{0})}{3}\right)\varepsilon^{2}+o(\varepsilon^{2})  \right).\label{eq5-16}
\end{equation}
Consequently,
\begin{align}
    \gamma_{\varepsilon}^{-\frac{2}{2^{\sharp}}}&=\omega_{n}^{\frac{2-n}{n}}f(x_{0})^{\frac{2-n}{n}}\left(1+\frac{1}{2n}\left( \frac{\Delta_{g}f(x_{0})}{f(x_{0})}+\frac{R_{g}(x_{0})}{3}\right)\varepsilon^{2}+o(\varepsilon^{2})\right).\label{eq5-17}
\end{align}
Finally, substituting \eqref{eq5-12} and \eqref{eq5-17} into \eqref{eq5-3}, and using the fact that
\[\frac{1}{K_{0}} = \frac{n(n - 2)}{4} \, \omega_n^{2/n},\]
we conclude that
\begin{equation*}
Q_{\varepsilon}=1+\frac{H(x_{0})}{2n(n-2)(n-2)(n-4)}\varepsilon^{2}+o(\varepsilon^{2}),
\end{equation*}
where 
\[H(x_{0})=(n-2)(n-4)\frac{\Delta_{g}f(x_{0})}{f(x_{0})}-2(n-2)R_{g}(x_{0})+8(n-1)\frac{b(x_{0})}{a(x_{0})}-(n^2-4)\frac{\Delta_{g}a(x_{0})}{a(x_{0})}.\]
Thanks to condition~\eqref{eq1-5} (see Theorem~\ref{Theorem1-1}), the quantity $H(x_{0})$ is strictly negative. Therefore
\begin{equation*}
Q_{\varepsilon}<1.
\end{equation*}

\subsection{ Taylor's expansion of $Q_{\protect\varepsilon}$ for $n=4$}

In this case, we use the same techniques as before. We note that the
expansion of $\gamma_{\varepsilon}$ remains the same as in the previous case.
We start with the leading term in $\mu_{\varepsilon}$ and take into account the following estimate (see Lemma~\ref{lemma5-2}): for any real numbers $p$, $q$ such that $p-q-1=0$, as $\varepsilon \to 0$, one has 
\begin{equation*}
\int_{0}^{(\frac{\delta }{\varepsilon })^{2}}\frac{t^{q}}{\left( 1+t\right)
^{p}}dt\sim \log \frac{1}{\varepsilon ^{2}}.
\end{equation*}
Therefore
\begin{align*}
\int_{M}a|\nabla
u_{\varepsilon}|^{2}\,dv_{g}&=\omega_{n-1}(n-2)^{2} 2^{n-3}\\[0.5em]
&\quad\times\bigg(I^{\frac{n}{2}}_{n}a(x_{0})-\left( \frac{\Delta_{g}a(x_{0})}{2n}+\frac{a(x_{0})}{6n}R_{g}(x_{0})\right)\varepsilon^{2}\log\frac{1}{\varepsilon^{2}}+o(\varepsilon^{2})  \bigg)\notag\\[0.5em]
&=\omega_{n-1}(n-2)^{2} 2^{n-3}I^{\frac{n}{2}}_{n}a(x_{0})\\[0.5em]
&\quad\times\bigg(1-\frac{1}{2nI^{\frac{n}{2}}_{n}}\left( \frac{\Delta_{g}a(x_{0})}{a(x_{0})}+\frac{R_{g}(x_{0})}{3}\right)\varepsilon^{2}\log\frac{1}{\varepsilon^{2}}+o(\varepsilon^{2})\bigg).
\end{align*}
Using the relations,
\[\omega_{n}=2^{n-1}I^{\frac{n}{2}-1}_{n}\omega_{n-1}\; \text{ and }\; I^{\frac{n}{2}}_{n}=\frac{n}{n-2}I^{\frac{n}{2}-1}_{n},\]
we deduce that
\begin{align}
\int_{M}a|\nabla
u_{\varepsilon}|^{2}\,dv_{g}=&\omega_{n}\frac{n(n-2)}{4}a(x_{0})\nonumber\\[0.5em]
&\times\bigg(1-\frac{1}{2nI^{\frac{n}{2}}_{n}}\left( \frac{\Delta_{g}a(x_{0})}{a(x_{0})}+\frac{R_{g}(x_{0})}{3}\right)\varepsilon^{2}\log\frac{1}{\varepsilon^{2}}+o(\varepsilon^{2})\bigg).\label{eq5-18}
\end{align}
For the remaining term, we have
\begin{equation}
    \int_{M}bv_{\varepsilon}^{2}\,dv_{g}
=2^{n-3}\omega_{n-1} \left(b(x_{0})\varepsilon^{2}\log\frac{1}{\varepsilon^{2}}+o(\varepsilon^{2})  \right).\label{eq5-19}
\end{equation}
Plugging \eqref{eq5-18} and \eqref{eq5-19} into \eqref{eq5-1}, we obtain \newpage
    \begin{align}
    \mu_{\varepsilon}&=\frac{n(n-2)\omega_{n}}{4}\, a(x_{0})\notag\\[0.5em]
    &\times\bigg\lbrace1+\frac{1}{6nI^{\frac{n}{2}}_{n}(n-2)^2}\bigg(6n\frac{b(x_{0})}{a(x_{0})}-(n-2)^2\bigg( \frac{3\Delta_{g}a(x_{0})}{a(x_{0})}+R_{g}(x_{0})\bigg)\bigg)\varepsilon^2\log\frac{1}{\varepsilon^2}+o(\varepsilon^2)\bigg\rbrace.
\label{eq5-20}
\end{align}
The expansion of $\gamma_{\varepsilon}$ remains the same as in the previous case. By substituting \eqref{eq5-12} and \eqref{eq5-20} into \eqref{eq5-3}, and using the fact that
\[\frac{1}{K_{0}} = \frac{n(n - 2)}{4} \, \omega_n^{2/n},\]
we finally obtain
\begin{equation*}
Q_{\varepsilon }=1+\frac{1}{6nI^{\frac{n}{2}}_{n}(n-2)^2}\bigg\lbrace6n\frac{b(x_{0})}{a(x_{0})}-(n-2)^2\bigg( \frac{3\Delta_{g}a(x_{0})}{a(x_{0})}+R_{g}(x_{0})\bigg)\bigg\rbrace\varepsilon^2\log\frac{1}{\varepsilon^2}+o(\varepsilon^2).
\end{equation*}
Thanks to condition~\eqref{eq1-5} (see Theorem~\ref{Theorem1-1}), for $n=4$, one has   
\[6n\frac{b(x_{0})}{a(x_{0})}-(n-2)^2\bigg( \frac{3\Delta_{g}a(x_{0})}{a(x_{0})}+R_{g}(x_{0})\bigg)=4\bigg(6\frac{b(x_{0})}{a(x_{0})}-\frac{3\Delta_{g}a(x_{0})}{a(x_{0})}+R_{g}(x_{0})\bigg)<0.\]
Consequently, 
\begin{equation*}
Q_{\varepsilon}<1.
\end{equation*}

\end{document}